%% file: LERW_polygon.tex
\pdfoutput=1
\documentclass[11pt]{article}
\usepackage{authblk}
\usepackage[toc,page]{appendix}

\usepackage{color}
\usepackage{helvet}         
\usepackage{courier}        
\usepackage{type1cm}        
\usepackage{framed}
\usepackage{tikz}
\usepackage{makeidx}         
\usepackage{graphicx}        
\usepackage{multicol}        
\usepackage[bottom]{footmisc}
\usepackage{amsmath}
\usepackage{amssymb}
\usepackage{bbold}
\usepackage{amsthm}
\usepackage{subcaption}
\usepackage{sidecap}
\usepackage{floatrow}
\usepackage{pdflscape}
\usepackage{enumitem}
\usepackage{comment}
\usepackage[font=small]{caption}
\usepackage[top=2cm, bottom=2cm, left=2cm, right=2cm]{geometry}
\newtheorem{theorem}{Theorem}
\newtheorem{corollary}[theorem]{Corollary}

\newtheorem{lemma}[theorem]{Lemma}

\newtheorem{proposition}[theorem]{Proposition}

\numberwithin{theorem}{section}
\numberwithin{figure}{section}
\numberwithin{equation}{section}

\DeclareMathOperator{\SLE}{SLE}

\begin{document}

\title{Loop-Erased Random Walk Branch of \\Uniform Spanning Tree in Topological Polygons}

\author{Mingchang Liu\thanks{liumc\_prob@163.com}}
\author{Hao Wu\thanks{hao.wu.proba@gmail.com. Funded by Beijing Natural Science Foundation (JQ20001, Z180003).}}
\affil{Tsinghua University, China}

\date{}

%
%
\maketitle
\begin{center}
\abstract{
We consider uniform spanning tree (UST) in topological polygons with $2N$ marked points on the boundary with alternating boundary conditions. In~\cite{LiuPeltolaWuUST}, the authors derive the scaling limit of the Peano curve in the UST. They are variants of SLE$_8$. In this article, we derive the scaling limit of the loop-erased random walk branch (LERW) in the UST. They are variants of SLE$_2$. The conclusion is a generalization of~\cite[Theorem~1.6]{HanLiuWuUST} where the authors derive the scaling limit of the LERW branch of UST when $N=2$. When $N=2$, the limiting law is SLE$_2(-1,-1; -1, -1)$. However, the limiting law is nolonger in the family of SLE$_2(\rho)$ process as long as $N\ge 3$. 
}
\end{center}

\newcommand{\eps}{\epsilon}
\newcommand{\ov}{\overline}
\newcommand{\U}{\mathbb{U}}
\newcommand{\T}{\mathbb{T}}
\newcommand{\HH}{\mathbb{H}}
\newcommand{\LA}{\mathcal{A}}
\newcommand{\LB}{\mathcal{B}}
\newcommand{\LC}{\mathcal{C}}
\newcommand{\LD}{\mathcal{D}}
\newcommand{\LF}{\mathcal{F}}
\newcommand{\LK}{\mathcal{K}}
\newcommand{\LE}{\mathcal{E}}
\newcommand{\LG}{\mathcal{G}}
\newcommand{\LL}{\mathcal{L}}
\newcommand{\LM}{\mathcal{M}}
\newcommand{\LQ}{\mathcal{Q}}
\newcommand{\LP}{\mathcal{P}}
\newcommand{\LR}{\mathcal{R}}
\newcommand{\LT}{\mathcal{T}}
\newcommand{\LS}{\mathcal{S}}
\newcommand{\LU}{\mathcal{U}}
\newcommand{\LV}{\mathcal{V}}
\newcommand{\LX}{\mathcal{X}}
\newcommand{\LY}{\mathcal{Y}}
\newcommand{\PartF}{\mathcal{Z}}
\newcommand{\LH}{\mathcal{H}}
\newcommand{\R}{\mathbb{R}}
\newcommand{\C}{\mathbb{C}}
\newcommand{\N}{\mathbb{N}}
\newcommand{\Z}{\mathbb{Z}}
\newcommand{\E}{\mathbb{E}}
\newcommand{\PP}{\mathbb{P}}
\newcommand{\QQ}{\mathbb{Q}}
\newcommand{\A}{\mathbb{A}}
\newcommand{\one}{\mathbb{1}}
\newcommand{\bn}{\mathbf{n}}
\newcommand{\MR}{MR}
\newcommand{\cond}{\,|\,}
\newcommand{\la}{\langle}
\newcommand{\ra}{\rangle}
\newcommand{\tree}{\Upsilon}
\newcommand{\prob}{\mathbb{P}}
\renewcommand{\Im}{\mathrm{Im}}
\renewcommand{\Re}{\mathrm{Re}}
\newcommand{\ii}{\mathfrak{i}}
\newcommand{\by}{\boldsymbol{y}}
\global\long\def\ud{\mathrm{d}}

\section{Introduction}
\input{tex/intro}
\section{Loop-erased random walk branch: the starting point}
\label{sec::startingpoint}
\input{tex/lerw}
\section{Loop-erased random walk branch: the Loewner chain}
\label{sec::loewnerchain}
\input{tex/convergence}

\end{document}

%% file: tex/intro.tex
We derive the scaling limit of loop-erased random walk branch (LERW) of uniform spanning tree (UST) in topological polygons. 
A (topological) polygon $(\Omega; x_1, \ldots, x_p)$ is a bounded simply connected domain $\Omega\subset\C$ with distinct boundary points $x_1, \ldots, x_p$ in counterclockwise order. We always assume $\partial\Omega$ is $C^1$ and simple. We denote by $(x_1x_2)$ the boundary arc between $x_1$ and $x_2$ in counterclockwise order.  
In this article, we focus on polygons with even number of marked points on the boundary $(\Omega; x_1, \ldots, x_{2N})$ with $N\ge 1$. 
Suppose $(\Omega^{\delta}; x_1^{\delta}, \ldots, x_{2N}^{\delta})$ is an approximation of $(\Omega; x_1, \ldots, x_{2N})$ on $\delta\Z^2$. 
We consider uniform spanning tree (UST) on $\Omega^{\delta}$ with alternating boundary conditions: the edges in the boundary arcs $(x_{2i-1}^{\delta}x_{2i}^{\delta})$ are forced to be contained in the tree for $i\in\{1, \ldots, N\}$. 
There are finitely many such trees, we consider uniform distribution on these trees. 
Let $\LT_{\delta}$ be a uniformly chosen tree on $\Omega^{\delta}$ with such alternating boundary conditions. There exists a branch $\gamma_{\delta}$ in $\LT_{\delta}$ connecting the arc $(x_1^{\delta}x_2^{\delta})$ to the arc $(x_{2N-1}^{\delta}x_{2N}^{\delta})$. The goal is to derive the limiting law of $\gamma_{\delta}$. In order to properly state the conclusion, we first introduce the following notions.

We first introduce the convergence of polygons. A curve is defined by a continuous map from $[0, 1]$ to $\C$. Let $\mathcal{C}$ be the space of unparameterized curves in $\C$.  Define the metric on $\mathcal{C}$ as follows:
\begin{align}\label{eqn::curves_metric}
d(\gamma_1, \gamma_2):=\inf\sup_{t\in[0,1]}\left|\hat{\gamma}_1(t)-\hat{\gamma}_2(t)\right|,
\end{align}
where the infimum is taken over all the choices of  parameterizations  $\hat{\gamma}_1$ and $\hat{\gamma}_2$ of $\gamma_1$ and $\gamma_2$. 
Fix a polygon $(\Omega;x_1,\ldots,x_{2N})$. We consider a sequence of discrete polygons $(\Omega^\delta;x_1^\delta,\ldots,x_{2N}^\delta)$ on $\delta\Z^2$ converge to $(\Omega;x_1,\ldots,x_{2N})$ in the following sense:  
\begin{align}\label{eqn::convpoly}
(x_i^\delta x_{i+1}^\delta)\text{ converges to }(x_ix_{i+1})\text{ in the metric}~\eqref{eqn::curves_metric},\quad\text{as }\delta\to 0,\quad\text{for }1\le i\le 2N.
\end{align}

Next, we introduce a particular conformal map $\phi$ on $\Omega$ which plays an essential role in this article. 

\begin{lemma}{\cite[Lemma~4.8]{LiuPeltolaWuUST}}
\label{lem::uniformizingcm}
For a polygon $(\Omega; x_1, \ldots, x_{2N})$, there exists a unique conformal map $\phi$ from $\Omega$ onto a rectangle of unit width with horizontal slits such that it maps the four points $(x_1, x_2, x_{2N-1}, x_{2N})$ to the four corners of the rectangle with $\phi(x_1)=0$ and it maps $(x_{2i-1}x_{2i})$ to horizontal slits for $i\in\{2,\ldots, N-1\}$. We denote by $K:=\Im\phi(x_{2N})$ the height of the rectangle. See Figure~\ref{fig::slitrectangle}. 
\end{lemma}
\begin{figure}[ht!]
\begin{center}
\includegraphics[width=0.6\textwidth]{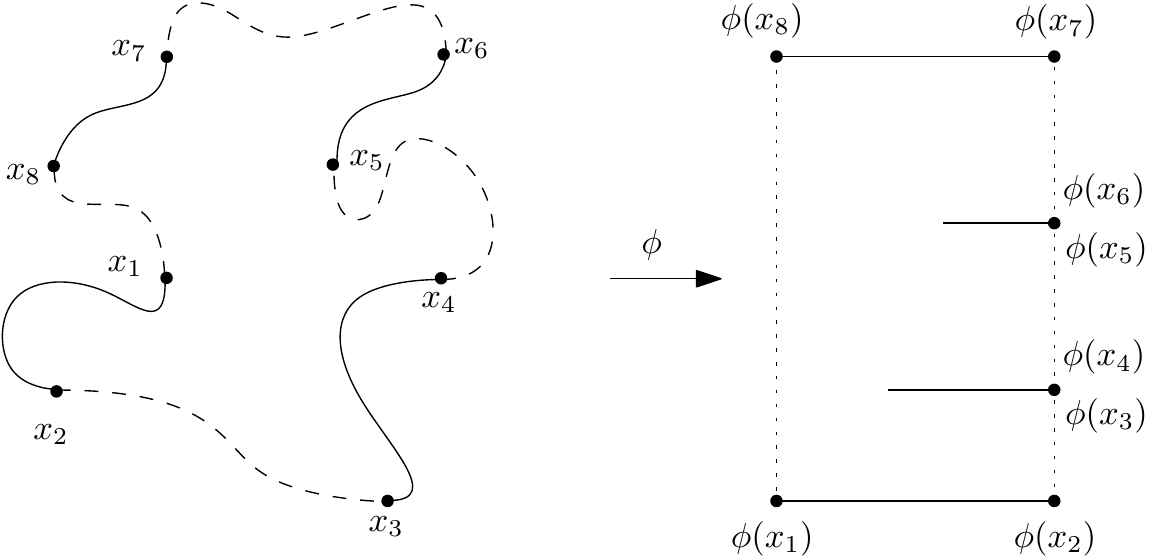}
\end{center}
\caption{\label{fig::slitrectangle} For a polygon $(\Omega; x_1, \ldots, x_{2N})$, there exists a unique conformal map $\phi$ from $\Omega$ onto a rectangle of unit width with horizontal slits as in Lemma~\ref{lem::uniformizingcm}. 
If $\Omega=\HH$ with marked points $y_1<\cdots<y_{2N}$, we denote the corresponding conformal map by $f(\cdot; y_1, \ldots, y_{2N})$. }
\end{figure}

When $\Omega=\HH$ with marked points $y_1<\cdots<y_{2N}$, we denote by $f(\cdot;y_1, \ldots, y_{2N})$ the conformal map in Lemma~\ref{lem::uniformizingcm}. We define partition function $\PartF$ as follows: for $y_1<w<y_2<\cdots<y_{2N}$, 
\begin{equation}\label{eqn::partitionfunction}
\PartF(w; y_1, \ldots, y_{2N}):=\partial_z f(z; y_1, \ldots, y_{2N})|_{z=w}. 
\end{equation}

Now, we are ready to state our conclusion. 

\begin{theorem}\label{thm::lerwconv}
Fix a polygon $(\Omega; x_1, \ldots, x_{2N})$ such that $\partial\Omega$ is $C^1$ and simple. 
Suppose that a sequence of polygons $(\Omega^{\delta}; x_1^{\delta}, \ldots, x_{2N}^{\delta})$ converges to $(\Omega; x_1, \ldots, x_{2N})$ as~\eqref{eqn::convpoly}. 
Consider the UST $\LT_{\delta}$ in $(\Omega^{\delta}; x_1^{\delta}, \ldots, x_{2N}^{\delta})$ with alternating boundary conditions and denote by $\gamma_{\delta}$ the branch in $\LT_{\delta}$ connecting $(x_1^{\delta}x_2^{\delta})$ to $(x_{2N-1}^{\delta}x_{2N}^{\delta})$ and we stop it when it hits $\cup_{i=2}^N(x_{2i-1}^{\delta}x_{2i}^{\delta})$. Then the law of $\gamma_{\delta}$ converges weakly to a continuous curve $\gamma$ in $(\Omega; x_1, \ldots, x_{2N})$ whose law is characterized by the following properties. Denote by $\LX=\gamma\cap (x_1x_2)$. 
\begin{enumerate}[label=(\arabic*)]
\item \label{item::thm1} Denote by $\phi$ the conformal map in Lemma~\ref{lem::uniformizingcm}. The law of $\phi(\LX)$ is uniform in $(0,1)$. 
\item \label{item::thm2} Let $\varphi$ be any conformal map from $\Omega$ onto $\HH$ such that $\varphi(x_1)<\cdots<\varphi(x_{2N})$. 
Given $\LX$, the conditional law of $\gamma$ is the image under $\varphi^{-1}$ of Loewner chain with the following driving function up to the first time that $\varphi(x_1)$ or $\varphi(x_2)$ is swallowed: 
\begin{equation}\label{eqn::lerw_loewner}
\begin{cases}
\ud W_t=\sqrt{2}\ud B_t+2(\partial_w\log \PartF)(W_t; V_t^1,\ldots,V_t^{2N})\ud t, \quad W_0=\varphi(\LX); \\
\ud V_t^j=\frac{2\ud t}{V_t^j-W_t},\quad V_0^j=\varphi(x_j), \quad j\in\{1, \ldots, 2N\},
\end{cases}
\end{equation}
where $\PartF$ is the partition function defined in~\eqref{eqn::partitionfunction} and $(B_t, t\ge 0)$ is standard one-dimensional Brownian motion. 
\end{enumerate}
\end{theorem} 

Let us briefly explain our strategy. 
\begin{itemize}
\item We derive the limiting law of the starting point  of the branch $\gamma_{\delta}$ in Section~\ref{sec::startingpoint}. From Wilson's algorithm~\cite{WilsonUSTLERW}, the law of $\gamma_{\delta}$ can be described by loop-erased random walk (LERW). Certain hitting probability of LERW is discrete harmonic function. We identify the limit of such discrete harmonic function and relate it to the holomorphic function $\phi$. 
\item We derive the conditional law of $\gamma$ given $\LX$ in Section~\ref{sec::loewnerchain}. To this end, we still consider certain hitting probability of LERW. We identify the limit of such discrete harmonic function as the imaginary part of a particular holomorphic function. The limit of the hitting probability provides a martingale observable for $\gamma$ from where we are able to solve the driving function of the Loewner chain. 
However, the corresponding hitting probability is a discrete harmonic function with complicated boundary conditions. The bulk part of the analysis in Section~\ref{sec::loewnerchain} is devoted to the analysis on the corresponding holomorphic function. 
The martingale observable in Section~\ref{sec::loewnerchain} is not the only reasonable one. It is possible that there are other observables which solve the question using a simpler analysis and provide distinct partition function $\tilde{\PartF}$ such that $\partial_w\log\tilde{\PartF}=\partial_w\log\PartF$. 
\end{itemize}

Theorem~\ref{thm::lerwconv} is a generalization of~\cite[Theorem~1.6]{HanLiuWuUST} where the authors derive the limiting law of $\gamma_{\delta}$ when $N=2$. 
When $N=2$, the conditional law of $\gamma$ given $\LX$ is $\SLE_2(-1,-1;-1,-1)$ in $\Omega$ from $\LX$ to $(x_3x_4)$ with force points $(x_4, x_1; x_2, x_3)$ stopped when it hits $(x_3x_4)$. However, the conditional law of $\gamma$ given $\LX$ is nolonger in the family of $\SLE_2(\rho)$ process when $N\ge 3$, see discussion at the end of Section~\ref{subsec::cvgloewner}. 

\medbreak
We end the introduction with a summary on previous conclusions about scaling limit of LERW in 2-dimension (as far as we know). They treat LERW in different setup from ours. 
\begin{itemize}
\item In~\cite[Theorem~1.1]{LawlerSchrammWernerLERWUST}, Lawler, Schramm and Werner consider LERW in simply connected domain starting from an interior point and stopped at the exit time. They prove that the scaling limit of such path is radial $\SLE_2$. 
\item In~\cite{ZhanLERW}, Zhan derives the scaling limit of LERW in finitely connected domains. Our setup is related to, but distinct from, examples described in~\cite[Section~4.2]{ZhanLERW}. We use the important idea from that paper about the Poisson kernel and we give a more concrete answer in our setup.
There are two main differences between~\cite{ZhanLERW} and our results: First, \cite{ZhanLERW} does not address the limiting law of the starting point of $\gamma_{\delta}$ as we do in Theorem~\ref{thm::lerwconv}~(1).  Second, our key analysis on the smoothness of the Poisson kernel in Section~\ref{sec::loewnerchain} has not been addressed in~\cite{ZhanLERW}. 
\item In~\cite{KenyonWilsonBoundaryPartitionsTreesDimers} and~\cite{KarrilaKytolaPeltolaCorrelationsLERWUST}, Kenyon and Wilson, and later Karrila, Kyt\"{o}l\"{a} and Peltola, derive connectivity probabilities of boundary branches in UST with fully wired boundary conditions. Based on these works, Karrila~\cite{KarrilaUSTBranches} considers the scaling limit of LERW branches in UST with fully wired boundary conditions. Fix $2N$ marked points on the boundary and consider the event that there exist $N$ LERW branches in the UST connecting among these $2N$ points. He shows that, conditioning on this rare event, the scaling limit of the $N$ LERW branches is multiple $\SLE_2$.  
\item In~\cite{ChelkakWanMassiveLERW}, the authors show that the scaling limit of ``massive loop-erased random walk" is the ``massive version" of $\SLE_2$. Although this result is not directly related to our setup, we do use their tool to analyze discrete harmonic functions. 
\end{itemize}

\noindent\textbf{Acknowledgements.} We thank Eveliina Peltola for helpful discussion. 

%% file: tex/lerw.tex
We assume the same setup as in Theorem~\ref{thm::lerwconv}. Recall that $\gamma_{\delta}$ is the branch in the UST in $(\Omega^{\delta}; x_1^{\delta}, \ldots, x_{2N}^{\delta})$ connecting $(x_1^{\delta}x_2^{\delta})$ to $(x_{2N-1}^{\delta}x_{2N}^{\delta})$ stopped when it hits $\cup_{i=2}^N(x_{2i-1}^{\delta}x_{2i}^{\delta})$. 
Recall that $\phi$ is the conformal map in Lemma~\ref{lem::uniformizingcm}. 
The goal of this section is to show the first conclusion in Theorem~\ref{thm::lerwconv}. 

\begin{proposition}\label{prop::startingpoint}
Define $\LX_{\delta}:=\gamma_{\delta}\cap (x^\delta_1x^\delta_2)$. Fix $a,b\in(x_1x_2)$ and $a^\delta,b^\delta\in(x^\delta_1x^\delta_2)$ such that the polygon $(\Omega^{\delta}; x_1^{\delta}, a^{\delta}, b^{\delta}, x_2^{\delta}, \ldots, x_{2N}^{\delta})$ converges to $(\Omega; x_1, a, b, x_2, \ldots, x_{2N})$ as~\eqref{eqn::convpoly}. Then, we have
\[\lim_{\delta\to 0}\PP\left[\LX_{\delta}\in(a^\delta b^\delta)\right]=\phi(b)-\phi(a).\]
\end{proposition}

Before the proof of Proposition~\ref{prop::startingpoint}, let us explain how it gives Theorem~\ref{thm::lerwconv}~\ref{item::thm1}. 
\begin{lemma}\label{lem::tight}
The family of curves $\{\gamma_\delta\}_{\delta>0}$ is tight. Moreover, suppose $\gamma$ is any subsequential limit of $\{\gamma_\delta\}_{\delta>0}$, then $\gamma$ intersects $\partial\Omega$ only at its two ends almost surely. 
\end{lemma}
\begin{proof}
The proof is the same as the proof of~\cite[Theorem 1.1]{LawlerSchrammWernerLERWUST}. See also~\cite[Lemma~4.11]{HanLiuWuUST}.
\end{proof}

\begin{corollary}\label{coro::LX}
Suppose $\gamma$ is any subsequential limit of $\{\gamma_\delta\}_{\delta>0}$. 
Define $\LX:=\gamma\cap(x_1x_2)$. The law of $\phi(\LX)$ is uniform in $(0,1)$.
\end{corollary}
\begin{proof}
From Lemma~\ref{lem::tight}, the law of $\LX_{\delta}$ converges to the law of $\LX$. Combining with Proposition~\ref{prop::startingpoint}, the law of $\phi(\LX)$ is uniform. 
\end{proof}

The rest of this section is devoted to the proof of Proposition~\ref{prop::startingpoint}. To this end, we first introduce the Poisson kernel in $\phi(\Omega)$ with mixed boundary conditions in Section~\ref{subsec::PoissonkernalR} and then prove Proposition~\ref{prop::startingpoint} in Section~\ref{subsec::startingpoint}.

\subsection{Poisson kernel with mixed boundary conditions}
\label{subsec::PoissonkernalR}

Recall from Lemma~\ref{lem::uniformizingcm} that $\phi(\Omega)$ is a rectangle with horizontal slits and $K=\Im\phi(x_{2N})$ is the height of the rectangle. We denote by $n$ the inner normal along $\partial\phi(\Omega)$. We define the Poisson kernel $P_{\phi(\Omega)}$ in $\phi(\Omega)$ with mixed boundary conditions as follows.
\begin{lemma}\label{lem::Pois}
For any fixed $w\in (0,1)\cup(\ii K,1+\ii K)$, there exists a unique function $P_{\phi(\Omega)}(\cdot,w):\phi(\Omega)\to \R$ which is positive and harmonic in $\phi(\Omega)$ and satisfies the following boundary conditions. 
\begin{enumerate}[label=(\arabic*)]
\item 
$P_{\phi(\Omega)}(\cdot,w)$ is continuous in $\overline{\phi(\Omega)}\setminus \{w\}$;
\item
$P_{\phi(\Omega)}(\cdot,w)=0$ on $\cup_{i=1}^{N}(\phi(x_{2i-1})\phi(x_{2i}))\setminus\{w\}$ and $\partial_n P_{\phi(\Omega)}(\cdot,w)=0$ on $\cup_{i=1}^{N}(\phi(x_{2i})\phi(x_{2i+1}))$ ;
\item
$P_{\phi(\Omega)}(z,w)-\frac{1}{\pi}\left|\Im \frac{1}{z-w}\right|$ is bounded in a neighborhood of $w$.
\end{enumerate}
\end{lemma}

\begin{proof} 
We prove the existence first.
Define $R$ to be the union of $\phi(\Omega)\cup(0,\ii K)\cup(1,1+\ii K)$ and its reflection with respect to $(0,\ii K)$. Define $\tilde R:=\cup_{m=-\infty}^{\infty}(R+2m)$. In this proof, we will always assume $z,u\in\tilde R$ and $w\in (-\infty,\infty)\times \{0, \ii K\}$. Denote by $G_{\tilde R}$ the Green function in $\tilde R$ with Dirichlet boundary conditon such that $G_{\tilde R}(z,u)+\frac{1}{2\pi}\log|z-u|$ is bounded when $z\to u$. The Poisson kernel in $\tilde R$ is defined as $P_{\tilde R}(z,w):=\partial_{n} G_{\tilde R}(z,u)|_{u=w}$. Define $S:=(-\infty,\infty)\times (0,\ii K)$, define the Green function $G_S$ and the Poisson kernel $P_S$ similarly. Note that for any fixed $u$, the function $G_S(\cdot,u)-G_{\tilde R}(\cdot,u)$ is a positive harmonic function in $\tilde R$. This implies that
\begin{equation}\label{eqn::1}
P_{\tilde R}(z,w)\le P_{S}(z,w). 
\end{equation}
Moreover, $P_{\tilde R}(z,w)$ has the same asymtotic property as $P_{S}(z,w)$ when $z\to w$.
Note that $P_{S}(z,w)=\frac{1}{K}\Im\frac{1}{1-\exp\left(\frac{\pi}{K}(z-w)\right)}$ when $w\in(0,1)$ and $P_{S}(z,w)=\frac{1}{K}\Im\frac{1}{\exp\left(\frac{\pi}{K}(z-w)\right)-1}$ when $w\in(\ii K,\ii K+1)$. This implies
\begin{equation}\label{eqn::asy}
P_{\tilde R}(z,w)-\frac{1}{\pi}\left|\Im \frac{1}{z-w}\right|\text{ is bounded in a neighborhood of }w.
\end{equation} 
By~\eqref{eqn::1} and the explicit form of $P_{S}$, we have $\sum_{m=-\infty}^{\infty}P_{\tilde R}(z+2m,w)+\sum_{m=-\infty}^{\infty}P_{\tilde R}(2m-\overline z,w)$ converges uniformly outside a neighborhood of $w$. Define
\begin{equation}
A(z,w):=\sum_{m=-\infty}^{\infty}P_{\tilde R}(z+2m,w)+\sum_{m=-\infty}^{\infty}P_{\tilde R}(2m-\overline z,w)\quad\text{and}\quad P_{\phi(\Omega)}(\cdot,w)=A(\cdot,w)|_{\phi(\Omega)}.
\end{equation}
Then, we have $A(\cdot,w)$ is harmonic in $\tilde R$. By definition, we have $A(z,w)=A(-\overline z,w)=A(2-\overline z,w)$. This implies 
\begin{equation}\label{eqn::partial}
\partial_n A(z,w)=0
\end{equation}
for all $z\in\cup_{i=1}^{N}(\phi(x_{2i})\phi(x_{2i+1}))$.
Combining the uniform convergence and~\eqref{eqn::partial}, $P_{\phi(\Omega)}$ satisfies the first two boundary conditions. Note that $P_{\phi(\Omega)}(z,w)$ has the same asymtotic property as $P_{\tilde R}(z,w)$ as $z\to w$. By~\eqref{eqn::asy}, we have that $P_{\phi(\Omega)}(\cdot,w)$ satisfies the third boundary condition. 

Next, we prove the uniqueness. Suppose $\tilde P_{\phi(\Omega)}(\cdot,\cdot)$ is another positive harmonic function which satisfies the same boundary conditions as $P_{\phi(\Omega)}$. Define $L=\tilde P_{\phi(\Omega)}-P_{\phi(\Omega)}$. Recall that $\varphi$ is any fixed conformal map from $\Omega$ onto $\HH$. We extend $L\circ\varphi^{-1}$ to $\C$ by reflecting $L\circ\varphi^{-1}$ with respect to $\R$ and we still denote the extension by $L\circ\varphi^{-1}$. Then, $L\circ\varphi^{-1}$ is bounded on $\C$ and is harmonic on $\C\setminus\cup_{i=1}^{N}(\varphi(x_{2i-1})\varphi(x_{2i}))$ and $L\circ\varphi^{-1}|_{\cup_{i=1}^{N}(\varphi(x_{2i-1}x_{2i}))}=0$. Maximum principle implies $L\circ\varphi^{-1}$ is identically zero. Thus, we have $P_{\phi(\Omega)}=\tilde P_{\phi(\Omega)}$. This completes the proof.
\end{proof}

The following properties of the Poisson kernel will be used later. 

\begin{lemma}\label{lem::prop}
For any $u,w\in (0,1)\cup (\ii K,\ii K+1)$, we have the followings. 
\begin{align}
\partial_n P_{\phi(\Omega)}(z,w)|_{z=u}=&\partial_n P_{\phi(\Omega)}(z,u)|_{z=w};\label{eqn::poissonkernel1}\\
\partial_{n} P_{\phi(\Omega)}(z,w)|_{z=u}>&0. \label{eqn::poissonkernel2}
\end{align}
\end{lemma}
\begin{proof}
Eq.~\eqref{eqn::poissonkernel1} follows from the construction of Poisson kernel and the symmetry of Green's function. 

For~\eqref{eqn::poissonkernel2}, we consider the standard Poisson kernel in $(0,1)\times(0,\ii K)$ and denote it by $P_{(0,1)\times(0,\ii K)}$. It has the same asymtotic property as $P_{\phi(\Omega)}$. Then, by the maximum principle, we have
\[P_{\phi(\Omega)}\ge P_{(0,1)\times(0,\ii K)}.\]
This implies that  
\[\partial_{n} P_{\phi(\Omega)}(z,w)|_{z=u}\ge \partial_{n} P_{(0,1)\times(0,\ii K)}(z,w)|_{z=u}>0.\]
This completes the proof.
\end{proof} 

\subsection{Proof of Proposition~\ref{prop::startingpoint}}
\label{subsec::startingpoint}

From Wilson's algorithm, we give an equivalent description of the branch $\gamma_{\delta}$. Define $\tilde\Omega^\delta$ to be the graph obtained from $\Omega$ by viewing each boundary arc $(x^\delta_{2i-1}x^\delta_{2i})$ as a vertex for $1\le i\le N$ and viewing different boundary arcs as different vertices. We sample a simple radom walk $\Gamma_\delta$ on $\tilde\Omega^\delta$, which starts from $(x_1^\delta x_2^\delta)$ and ends at $(x_{2N-1}^\delta x_{2N}^\delta)$. The loop-erased random walk $\tilde{\gamma}_\delta$ is defined as follows by induction. Define $\gamma_\delta(0):=(x_1^\delta x_2^\delta)$. Suppose for $1\le j\le k$, the step $\gamma_\delta(j)$ has been defined. Define $n_k:=\max\{m:\Gamma_\delta(m)=\gamma_\delta(k)\}$. Then, we define $\gamma_\delta(k+1):=\Gamma_\delta(n_k+1)$. We end the induction when $\tilde{\gamma}_\delta$ arrives at $(x_{2N-1}^{\delta}x_{2N}^{\delta})$. Note that $\tilde{\gamma}_{\delta}$ stopped at the first hitting time of $\cup_{i=2}^{N}(x^\delta_{2i-1}x^\delta_{2i})$ has the same law as $\gamma_{\delta}$ stopped at the same hitting time.

In the following lemma, we connect the simple random walk on $\tilde{\Omega}^{\delta}$ to the conformal map $\phi$ in Lemma~\ref{lem::uniformizingcm}. Recall that 
$K:=\Im\phi(x_{2N})$ is the height of the rectangle $\phi(\Omega)$. 
\begin{lemma}\label{lem::harm}
For every $v^\delta\in\Omega^\delta$, we define $\lambda^\delta(v^\delta)$ to be the probability that a simple random walk on $\tilde\Omega^\delta$ starting from $v^\delta$ hits $(x^\delta_{2N-1}x^\delta_{2N})$ before $(x_1^\delta x_2^\delta)$ and regard $\lambda^\delta$ as a function on $\Omega^\delta$. The discrete harmonic function $\lambda^\delta$ converges to $\frac{1}{K}\Im \phi$ locally uniformly as $\delta\to 0$. 
\end{lemma}
\begin{proof}
Note that $\lambda^\delta$ is discrete harmonic on $\Omega^\delta\setminus\cup_{i=1}^{N}(x^\delta_{2i-1}x^\delta_{2i})$ and it is constant on each $(x_{2i-1}^\delta x_{2i}^\delta)$ for $1\le i\le N$. Moreover, $\lambda^\delta=0$ on $(x_1^\delta x_2^\delta)$ and $\lambda^\delta=1$ on $(x_{2N-1}^\delta x_{2N}^\delta)$. Denote by $r^\delta$ the discrete harmonic conjugate of $\lambda^\delta$ defined on the dual graph of $\Omega^{\delta}$, which equals $0$ on $(x^\delta_{2N}x^\delta_1)$. By the definition of $\lambda^\delta$, it is clear that $r^\delta$ equals the same positive constant on $(x^\delta_{2i}x^\delta_{2i+1})$ for $1\le i\le N-1$. 
We denote this constant by $M^{\delta}$. 

Define $\tilde{\phi}^{\delta}=r^\delta+\ii \lambda^\delta$. If $\{M^\delta\}_{\delta>0}$ is uniformly bounded, for any subsequence $\delta_n\to 0$, there exists a subsequence, still denoted by $\delta_n$, such that $\tilde{\phi}^{\delta_n}$ converges to an holomorphic function $\tilde\phi$ locally uniformly and $M^{\delta_n}$ converges to a positive constant $M$. By discrete Beurling estimate, we have 
\[\Re\tilde\phi|_{\cup_{i=1}^{N-1}(x_{2i}x_{2i+1})}=M,\quad \text{and}\quad\Im\tilde\phi|_{(x_{2N-1}x_{2N})}=1.\]
Thus, we have $\tilde\phi=M\phi$ and $M=\frac{1}{K}$ as desired. 

It remains to prove that $\{M^\delta\}_{\delta>0}$ is uniformly bounded. If this is not the case, suppose $M^{\delta_{n}}\to\infty$ as $n\to\infty$ for a subsequence $\{\delta_n\}_{n\ge 1}$ and consider $\{\frac{\tilde{\phi}^{\delta_n}}{M^{\delta_n}}\}_{n\ge 1}$. Then any subsequential limit $\tilde\phi$ is identically zero. But by discrete Beurling estimate, we have $\Re\tilde\phi|_{\cup_{i=1}^{N-1}(x_{2i}x_{2i+1})}=1$, which is a contradiction. This completes the proof.
\end{proof}
\begin{proof}[Proof of Proposition~\ref{prop::startingpoint}]
It suffices to prove for every $\tilde{a}, \tilde{b}\in(x_1x_2)$ such that $[ab]\cap [\tilde{a}\tilde{b}]=\emptyset$, choose $\tilde a^\delta, \tilde b^\delta\in\partial\Omega^\delta$ such that $(\tilde a^\delta\tilde b^\delta)$ converges to $(\tilde{a}\tilde{b})$ in metric~\eqref{eqn::curves_metric}, we have
\begin{equation}\label{eqn::startingpoint}
\lim_{\delta\to 0}\frac{\PP[\LX_{\delta}\in(a^\delta b^\delta)]}{\PP[\LX_{\delta}\in(\tilde a^\delta \tilde b^\delta)]}=\frac{\phi(b)-\phi(a)}{\phi(\tilde{b})-\phi(\tilde{a})}.
\end{equation}

We may assume $\tilde{a},\tilde{b}, a, b$ are in counterclockwise order. For any $w^\delta\in\Omega^\delta$, denote by $\PP^{w^\delta}$ the law of simple random walk in $\tilde\Omega^\delta$ starting from $w^\delta$. Then, for every $a^\delta, b^\delta\in(x_1^\delta x_2^\delta)$, we have
\[\frac{\PP[\LX_{\delta}=a^\delta]}{\PP[\LX_{\delta}=b^\delta]}=\frac{\sum_{\substack{v^\delta\sim a^\delta\\ v^\delta\notin\partial\Omega^\delta}}\PP^{v^\delta}[\text{SRW hits }(x_{2N-1}^\delta x_{2N}^\delta)\text{ before }(x_1^\delta x_2^\delta)]}{\sum_{\substack{w^\delta\sim b^\delta\\ w^\delta\notin\partial\Omega^\delta}}\PP^{w^\delta}[\text{SRW hits }(x_{2N-1}^\delta x_{2N}^\delta)\text{ before }(x_1^\delta x_2^\delta)]}.\]
Recall that $\lambda^\delta$ is defined in Lemma~\ref{lem::harm}. Then, we have
\begin{equation}\label{eqn::3}
\frac{\PP[\LX_{\delta}\in(a^\delta b^\delta)]}{\PP[\LX_{\delta}\in(\tilde a^\delta \tilde b^\delta)]}=\frac{\sum_{\substack{v^\delta\sim (a^\delta b^\delta)\\ v^\delta\notin\partial\Omega^\delta}}\lambda^\delta(v^\delta)}{\sum_{\substack{w^\delta\sim (\tilde a^\delta \tilde b^\delta)\\ w^\delta\notin\partial\Omega^\delta}}\lambda^\delta(w^\delta)},
\end{equation}
where the sum is taken over the endpoints of all egdes connecting to $(a^\delta b^\delta)$ or $(\tilde a^\delta \tilde b^\delta)$.

Now, we consider the division $\{w_0=\tilde{a}, w_1,\ldots,w_m=\tilde{b}\}$ of $(\tilde{a}\tilde{b})$ and the division $\{v_0=a,v_1,\ldots,v_n=b\}$ of $(ab)$ such that $d(w_j,w_{j+1})<s$ and $d(v_k,v_{k+1})<s$ for $0\le j\le m-1$ and $0\le k\le n-1$, where $s$ is a small constant which will be determined later. Suppose they are in counterclockwise order. Choose the corresponding divisions $\{w^\delta_0=\tilde{a}^{\delta},w^\delta_1,\ldots,w^\delta_m=\tilde{b}^{\delta}\}$ of $(\tilde a^\delta\tilde b^\delta)$ and the division $\{v^\delta_0=a^\delta,v^\delta_1,\ldots,v^\delta_n=b^\delta\}$ of $(a^\delta b^\delta)$ on $\partial\Omega^\delta$. In this proof, we define $\partial\Omega^\delta:=\cup_{i=1}^{N}(x_{2i-1}^\delta x_{2i}^{\delta})$ and $\partial\Omega:=\cup_{i=1}^{N}(x_{2i-1}x_{2i})$. 
Define 
\[\tilde h_j^\delta(z^\delta):=\PP^{z^\delta}[\text{SRW hits }\partial\Omega^\delta\text{ at }(w_j^\delta w_{j+1}^\delta)]\]
and
\[h_k^\delta(z^\delta):=\PP^{z^\delta}[\text{SRW hits }\partial\Omega^\delta\text{ at }(v^\delta_k v^\delta_{k+1})]\]
 for $0\le j\le m-1$ and $0\le k\le n-1$.
Define $\tilde h_j$ to be the harmonic function in $\Omega$ with the following boudary condition: $\tilde h_j=1$ on $(w_jw_{j+1})$ and $\tilde h_j=0$ on $\partial\Omega\setminus(w_jw_{j+1})$ and $\partial_n \tilde h_j\circ\phi^{-1}=0$ on $\cup_{i=1}^{N}(\phi(x_{2i})\phi(x_{2i+1}))$. Define $h_k$ to be the harmonic function in $\Omega$ with the following boudary condition: $h_k=1$ on $(v_kv_{k+1})$ and $h_k=0$ on $\partial\Omega\setminus(w_jw_{j+1})$ and $\partial_n h_k\circ\phi^{-1}=0$ on $\cup_{i=1}^{N}(\phi(x_{2i})\phi(x_{2i+1}))$. By the same proof of Lemma~\ref{lem::harm}, we can prove that $\tilde h_j^\delta$ converges to $\tilde h_j$ and $h_k^\delta$ converges to $h_k$ locally uniformly.
By~\cite[Corollary~3.8]{ChelkakWanMassiveLERW}, for every $\epsilon>0$, there exists $s_\epsilon>0$ such that for every $v\in \Omega$ such that $d(v,(v_k v_{k+1}))<s_\epsilon$ and $v^\delta\in\Omega^\delta\setminus\partial\Omega^\delta$ such that $d(v^\delta,v)<s_\epsilon$, we have
\begin{equation}\label{eqn::4}
(1-\epsilon)\frac{1}{K}\frac{\phi(v)}{\tilde h_j(v)}\le\frac{\lambda^\delta(v^\delta)}{\tilde h_j^\delta(v^\delta)}\le (1+\epsilon)\frac{1}{K}\frac{\phi(v)}{\tilde h_j(v)}.
\end{equation}
By letting $v\to v_k$ with $\Re\phi(v)=\Re\phi(v_k)$, we have 
\begin{equation}\label{eqn::5}
\frac{\phi(v)}{\tilde h_j(v)}\to \frac{1}{\partial_n\left(\tilde h_j\circ\phi^{-1}\right)(\phi(v_k))}.
\end{equation}
Note that 
\[h_j\circ\phi^{-1}(\cdot)=\int_{\phi(w_j)}^{\phi(w_{j+1})}P_{\phi(\Omega)}(\cdot,w)\ud w.\]
By~\eqref{eqn::poissonkernel2}, we have 
\[\partial_n P_{\phi(\Omega)}(u,\phi(w_j))|_{u=\phi(v_k)}>0.\]
Thus, for every $\epsilon>0$, there exists $s'_\epsilon$, such that if $d(w_j,w_{j+1})<s'_\epsilon$, we have
\begin{equation}
1-\epsilon\le\frac{\partial_n\left(\tilde h_j\circ\phi^{-1}\right)(\phi(v_k))}{(\phi(w_{j+1})-\phi(w_{j}))\partial_n P_{\phi(\Omega)}(u,\phi(w_j))|_{u=\phi(v_k)}}\le 1+\epsilon.
\end{equation}
Now we choose $s<s_\epsilon\wedge s'_\epsilon$. Combining~\eqref{eqn::3},~\eqref{eqn::4} and~\eqref{eqn::5}, for every $v^\delta\sim (v^\delta_k v^\delta_{k+1})$ and $v^\delta\notin\partial\Omega^\delta$, we have
\begin{equation}\label{eqn::6}
(1-\epsilon)^2\le K\frac{\lambda^\delta(v^\delta)}{\tilde h_j^\delta(v^\delta)}(\phi(w_{j+1})-\phi(w_{j}))\partial_n P_{\phi(\Omega)}(u,\phi(w_j))|_{u=\phi(v_k)}\le (1+\epsilon)^2
\end{equation}
This implies
\begin{equation}\label{eqn::7}
(1-\epsilon)^2\le K\frac{\sum_{\substack{v^\delta\sim (v_k^\delta v_{k+1}^\delta)\\ v^\delta\notin\partial\Omega^\delta}}\lambda^\delta(v^\delta)(\phi(w_{j+1})-\phi(w_{j}))\partial_n P_{\phi(\Omega)}(u,\phi(w_j))|_{u=\phi(v_k)}}{\sum_{\substack{v^\delta\sim (v_k^\delta v_{k+1}^\delta)\\ v^\delta\notin\partial\Omega^\delta}}\tilde h_j^\delta(v^\delta)}\le (1+\epsilon)^2.
\end{equation}
Similarly, we have
\begin{equation}\label{eqn::8}
(1-\epsilon)^2\le K\frac{\sum_{\substack{w^\delta\sim (w_j^\delta w_{j+1}^\delta)\\ w^\delta\notin\partial\Omega^\delta}}\lambda^\delta(w^\delta)(\phi(v_{k+1})-\phi(v_{k}))\partial_n P_{\phi(\Omega)}(u,\phi(v_k))|_{u=\phi(w_j)}}{\sum_{\substack{w^\delta\sim (w_j^\delta w_{j+1}^\delta)\\ w^\delta\notin\partial\Omega^\delta}} h_k^\delta(w^\delta)}\le (1+\epsilon)^2.
\end{equation}
By definition of $\tilde h^\delta_j$ and $h^\delta_k$ and the reversibility of random walk, we have
\begin{equation}\label{eqn::9}
\sum_{\substack{v^\delta\sim (v_k^\delta v_{k+1}^\delta)\\ v^\delta\notin\partial\Omega^\delta}}\tilde h_j^\delta(v^\delta)=\sum_{\substack{w^\delta\sim (w_j^\delta w_{j+1}^\delta)\\ w^\delta\notin\partial\Omega^\delta}} h_k^\delta(w^\delta).
\end{equation}
Combining~\eqref{eqn::poissonkernel1} and~\eqref{eqn::9}, we have
\begin{equation}
(1-\epsilon)^4\frac{\phi(v_{k+1})-\phi(v_k)}{\phi(w_{j+1})-\phi(w_j)}\le\frac{\sum_{\substack{v^\delta\sim (v_k^\delta v_{k+1}^\delta)\\ v^\delta\notin\partial\Omega^\delta}}\lambda^\delta(v^\delta)}{\sum_{\substack{w^\delta\sim (w_j^\delta w_{j+1}^\delta)\\ w^\delta\notin\partial\Omega^\delta}}\lambda^\delta(w^\delta)}\le(1+\epsilon)^4\frac{\phi(v_{k+1})-\phi(v_k)}{\phi(w_{j+1})-\phi(w_j)}.
\end{equation}
Thus, we have
\begin{equation}
(1-\epsilon)^4\frac{\phi(b)-\phi(a)}{\phi(\tilde{b})-\phi(\tilde{a})}\le\frac{\sum_{\substack{v^\delta\sim (a^\delta b^\delta)\\ v^\delta\notin\partial\Omega^\delta}}\lambda^\delta(v^\delta)}{\sum_{\substack{w^\delta\sim (\tilde a^\delta \tilde b^\delta)\\ w^\delta\notin\partial\Omega^\delta}}\lambda^\delta(w^\delta)}\le(1+\epsilon)^4\frac{\phi(b)-\phi(a)}{\phi(\tilde{b})-\phi(\tilde{a})}
\end{equation}
By letting $\delta\to 0$ and $\epsilon\to 0$, we get~\eqref{eqn::startingpoint}. This completes the proof.
\end{proof}

%% file: tex/convergence.tex
We assume the same setup as in Theorem~\ref{thm::lerwconv}. Recall that $\gamma_{\delta}$ is the branch in the UST in $(\Omega^{\delta}; x_1^{\delta}, \ldots, x_{2N}^{\delta})$ connecting $(x_1^{\delta}x_2^{\delta})$ to $(x_{2N-1}^{\delta}x_{2N}^{\delta})$ stopped when it hits $\cup_{i=2}^N(x_{2i-1}^{\delta}x_{2i}^{\delta})$. From Lemma~\ref{lem::tight}, the family $\{\gamma_{\delta}\}_{\delta>0}$ is tight. Suppose $\gamma$ is any subsequental limit and denote by $\LX=\gamma\cap(x_1x_2)$. Recall that $\phi$ is the conformal map in Lemma~\ref{lem::uniformizingcm}. 
From Corollary~\ref{coro::LX}, the law of $\phi(\LX)$ is uniform in $(0,1)$. The goal of this section is to derive the law of $\gamma$ given $\LX$, i.e. to show Theorem~\ref{thm::lerwconv}~\ref{item::thm2}. 

\begin{proposition}\label{prop::lawofgamma}
Let $\varphi$ be any conformal map from $\Omega$ onto $\HH$ such that $\varphi(x_1)<\cdots<\varphi(x_{2N})$. 
Given $\LX$, the conditional law of $\gamma$ is the image under $\varphi^{-1}$ of Loewner chain with driving function given by~\eqref{eqn::lerw_loewner} up to the first time that $\varphi(x_1)$ or $\varphi(x_2)$ is swallowed.  
\end{proposition}

This section is organized as follows. We first give prelimiaries on Loewner chain and on the uniformizing conformal map $f$ in Section~\ref{subsec::pre}. Then we introduce and analyze an holomorphic observable in Sections~\ref{subsec::observable} and~\ref{subsec::holomorphic}. 
Finally, we derive the law of $\gamma$ using the observable and complete the proof of Proposition~\ref{prop::lawofgamma} in Section~\ref{subsec::cvgloewner}. 

\subsection{Preliminaries}
\label{subsec::pre}

We first collect basic notions on the Loewner chain. 
We call a compact subset $K$ of $\overline{\HH}$ an $\HH$-hull if $\HH\setminus K$ is simply connected. Riemann's Mapping Theorem asserts that there exists a unique conformal map $g_K$ from $\HH\setminus K$ onto $\HH$ such that
$\lim_{z\to\infty}|g_K(z)-z|=0$.
We call such $g_K$ the conformal map from $\HH\setminus K$ onto $\HH$ normalized at $\infty$. 

Loewner chain is a collection of $\HH$-hulls $(K_{t}, t\ge 0)$ associated with the family of conformal maps $(g_{t}, t\ge 0)$ obtained by solving the Loewner equation: for each $z\in\mathbb{H}$,
\[\partial_{t}{g}_{t}(z)=\frac{2}{g_{t}(z)-W_{t}}, \quad g_{0}(z)=z,\]
where $(W_t, t\ge 0)$ is a one-dimensional continuous function which we call the driving function. Let $T_z$ be the swallowing time of $z$ defined as $\sup\{t\ge 0: \min_{s\in[0,t]}|g_{s}(z)-W_{s}|>0\}$.
Let $K_{t}:=\overline{\{z\in\mathbb{H}: T_{z}\le t\}}$. Then $g_{t}$ is the unique conformal map from $H_{t}:=\mathbb{H}\backslash K_{t}$ onto $\mathbb{H}$ normalized at $\infty$. 
We say that $(K_t, t\ge 0)$ can be generated by a continuous curve $(\eta(t), t\ge 0)$ if for any $t$, the unbounded connected component of $\HH\setminus\eta[0,t]$ coincides with $H_t=\HH\setminus K_t$. 

Schramm Loewner Evolution $\SLE_{\kappa}$ is the random Loewner chain $(K_{t}, t\ge 0)$ driven by $W_t=\sqrt{\kappa}B_t$ where $(B_t, t\ge 0)$ is a standard one-dimensional Brownian motion.
In this case, the Loewner chain is generated by continuous curve. In this article, we focus on $\kappa=2$.

\medbreak

Next, we give preliminaries on the conform map $f$ in Figure~\ref{fig::slitrectangle}. 
Recall that for $\HH$ with marked points $y_1<\cdots<y_{2N}$, we denote by $f(\cdot)=f(\cdot; y_1, \ldots, y_{2N})$ the conformal map from $\HH$ onto a rectangle of unit width with horizontal slits such that it maps $(y_1, y_2, y_{2N-1}, y_{2N})$ to the four corners of the rectangle with $f(y_1)=0$ and it maps $(y_{2i-1}y_{2i})$ to horizontal slit for $i\in\{2, \ldots, N-1\}$. 
From Schwarz-Christorffel formula, the conformal map $f$ has the following form: 
\begin{equation}\label{eqn::conformalmap_slitrect}
f(z)=f(z; y_1, \ldots, y_{2N}):=\int_{y_1}^z \frac{\prod_{\ell=1}^{N-2}(u-\mu^{(\ell)})\ud u}{\prod_{j=1}^{2N}(u-y_j)^{1/2}} \left(\int_{y_1}^{y_2} \frac{\prod_{\ell=1}^{N-2}(u-\mu^{(\ell)})\ud u}{\prod_{j=1}^{2N}(u-y_j)^{1/2}} \right)^{-1}, \quad z\in\HH,
\end{equation}
where $\mu^{(1)}, \ldots, \mu^{(N-2)}$ are well-chosen parameters which will be specified below. 

Define 
\begin{equation}\label{eqn::parametera_polynomial}
\omega_{r}:=\frac{u^{r}\ud u}{\prod_{j=1}^{2N}(u-y_{j})^{1/2}},\quad \text{for }0\le r\le N-2.
\end{equation}
We choose a branch for each $\omega_r$ simultaneously positive on 
\begin{equation}\label{eqn::branchchoice}
\{u\in\C\setminus\{y_1, \ldots, y_{2N}\}: u\in\R \text{ and } y_{2N}<u\}.
\end{equation}
Let $R=\{R(i,j): 1\le i, j\le N-2\}$ be the matrix given by 
\begin{equation}\label{eqn::parametera_matrix}
R(i, j)=\int_{y_{2i+1}}^{y_{2i+2}}\omega_{j-1}, \quad \text{for }1\le i, j\le N-2. 
\end{equation}

\begin{lemma}\cite[Lemma~4.9]{LiuPeltolaWuUST}. \label{lem::parameters}
The matrix $R$ in~\eqref{eqn::parametera_matrix} is invertible. Define 
\begin{equation*}
(\nu_{N-2}, \ldots, \nu_1)^t: =-R^{-1}\left(\int_{y_3}^{y_4}\omega_{N-2},\,\cdots,\,\int_{y_{2N-3}}^{y_{2N-2}}\omega_{N-2}\right)^{t}.
\end{equation*}
Define the polygnomial 
\begin{equation*}
Q(w):=w^{N-2}+\sum_{n=0}^{N-3}\nu_{N-2-n} w^{n}.
\end{equation*} 
Then the parameters $\mu^{(1)}, \ldots, \mu^{(N-2)}$ in~\eqref{eqn::conformalmap_slitrect} are the $N-2$ roots of $Q$. 
\end{lemma}

In this section, we denote by $\by=(y_1, \ldots, y_{2N})$ to simplify notations. Consequently, $f(z; y_1, \ldots, y_{2N})$ will be denoted by $f(z; \by)$. Recall from~\eqref{eqn::partitionfunction} that, for $y_1<w<y_2<\cdots<y_{2N}$, we denote 
\[\PartF(w; \by)=\partial_z f(z; \by)|_{z=w}.\]
\subsection{Observable}
\label{subsec::observable}
Let $\varphi$ be any conformal map from $\Omega$ onto $\HH$ such that $y_1:=\varphi(x_1)<\cdots<y_{2N}:=\varphi(x_{2N})$. 
Fix $(ab)\subset(x_1x_2)$ and $(cd)\subset(x_{2N-1}x_{2N})$. Suppose $(a^\delta b^\delta)\subset(x^\delta_1x^\delta_2)$ and $(c^\delta d^\delta)\subset(x^\delta_{2N-1}x^\delta_{2N})$ such that $(\Omega^{\delta}; x_1^{\delta}, a^\delta, b^\delta, x_2^{\delta}, \ldots, x_{2N-1}^{\delta}, c^{\delta}, d^{\delta}, x_{2N}^{\delta})$ converges to $(\Omega; x_1, a, b, x_2, \ldots, x_{2N-1}, c,d,x_{2N})$ as~\eqref{eqn::convpoly}.

Recall that we have given an equivalent description of the branch $\gamma_{\delta}$ by simple random on $\tilde\Omega^\delta$ at the beginning of Section~\ref{subsec::startingpoint}. 
Define $\lambda_{c^\delta, d^\delta}^\delta(v^\delta)$ to be the probability that a simple random walk on $\tilde\Omega^\delta$ starting from $v^\delta$ hits $(x_1^\delta x_2^\delta)\cup(x^\delta_{2N-1}x^\delta_{2N})$ at $(c^\delta d^\delta)$ and regard $\lambda^\delta$ as a function on $\Omega^\delta$. The first goal is to derive the limit of the harmonic function $\lambda^{\delta}_{c^{\delta}, d^{\delta}}$. From the boundary conditions of $\lambda^{\delta}_{c^{\delta}, d^{\delta}}$, its limit should be related to the following holomorphic function. 

\begin{lemma}\label{lem::probuniqueness}
There exists a unique holomorphic function $h$ on $\HH$, which satisfies the following boundary conditions:
\begin{enumerate}[label=(\arabic*)]
\item
$\Re h=0$ on $(y_{2N},+\infty)\cup (-\infty,y_1)$; 
\item  $\Re h$ equals the same constant on $\cup_{i=1}^{N-1}(y_{2i},y_{2i+1})$; 
\item
$\Im h$ is constant on $(y_{2i-1},y_{2i})$ for $1\le i\le N-1$; 
\item $\Im h=0$ on $(y_1,y_2)\cup (y_{2N-1},\varphi(c))\cup(\varphi(d),y_{2N})$ and $\Im h=1$ on $(\varphi(c),\varphi(d))$;
\item
$\Im h$ is bounded on $\overline\HH$ and continuous on $\overline\HH\setminus\{\varphi(c),\varphi(d)\}$.
\end{enumerate}
\end{lemma}

Assuming this is true, we obtain the limit of $\lambda^{\delta}_{c^{\delta}, d^{\delta}}$. 

\begin{lemma}\label{lem::convobservable}
Let $h$ be the holomorphic function in Lemma~\ref{lem::probuniqueness}. 
The sequence of discrete harmonic functions $\lambda_{c^\delta, d^\delta}^\delta$ converges to $\Im h$ locally uniformly as $\delta\to 0$.
\end{lemma}
\begin{proof}
This can be proved by the same argument as in the proof of Lemma~\ref{lem::harm} where we need to replace Lemma~\ref{lem::uniformizingcm} by Lemma~\ref{lem::probuniqueness}. 
\end{proof}

Consider the random walk on $\tilde\Omega^\delta$ which starts from $(x^\delta_1x^\delta_2)$ and ends at $(x^\delta_{2N-1}x^\delta_{2N})$. Denote by $\tilde\gamma_\delta$ the loop-erased path obtained from this random walk, as described at the beginning of Section~\ref{subsec::startingpoint}. Define 
\begin{equation}\label{eqn::def_twoends}
\LX_\delta:=\tilde\gamma_\delta\cap(x^\delta_1x^\delta_2), 
\quad
\LY_\delta:=\tilde\gamma_\delta\cap(x^\delta_{2N-1}x^\delta_{2N}).
\end{equation}
We will derive the joint distribution of $(\LX_{\delta}, \LY_{\delta})$. Recall that $K(\by)$ is the height of the rectangle $f(\HH)$.
Let $h$ be the holomorphic function in Lemma~\ref{lem::probuniqueness}. 
For $x\in (y_{2N-1},y_{2N})$, the following limit exists:  
\begin{equation}\label{def::PH}
P_{\HH}(z,x;\by):=K(\by)\lim_{s,t\to x}\frac{\Im h(z, t, s;\by)}{s-t}.
\end{equation}

The limiting joint distribution of $(\LX_{\delta}, \LY_{\delta})$ is given by $P_{\HH}$. 

\begin{proposition}\label{prop::jointdis}
We have
\[\lim_{\delta\to 0}\PP[\LX_\delta\in (a^\delta b^\delta), \LY_\delta\in (c^\delta d^\delta)]=\int_{\varphi(c)}^{\varphi(d)}\ud x\int_{\varphi(a)}^{\varphi(b)}\ud w\;\partial_n P_{\HH}(z,x;\by)|_{z=w},\]
where $P_{\HH}$ is defined in~\eqref{def::PH}. 
\end{proposition}

This joint distribution will give a martingale observable which will be used to complete the proof of Proposition~\ref{prop::lawofgamma} in Section~\ref{subsec::cvgloewner}. 

\begin{proof}
Recall that $\lambda^{\delta}$ is the discrete harmonic function defined in Lemma~\ref{lem::harm}. Note that 
\[\PP\left[\LY_\delta\in(c^\delta d^\delta)\cond\LX_\delta\right]=\frac{\lambda_{c^\delta,d^\delta}^{\delta}(\LX_{\delta})}{\lambda^\delta(\LX_{\delta})}.\]
By~\cite[Corollary~3.8]{ChelkakWanMassiveLERW}, for every $\epsilon>0$, there exists $s(\epsilon)>0$, such that for every $v^\delta\in\Omega^\delta$ wuch that $d(v^\delta,\LX_\delta)<s(\epsilon)$, we have
\[(1-\epsilon)\frac{\lambda_{c^\delta,d^\delta}^{\delta}(v^\delta)}{\lambda^\delta(v^\delta)}\le\frac{\lambda_{c^\delta,d^\delta}^{\delta}(\LX_{\delta})}{\lambda^\delta(\LX_{\delta})}\le (1+\epsilon)\frac{\lambda_{c^\delta,d^\delta}^{\delta}(v^\delta)}{\lambda^\delta(v^\delta)}.\]
Choose $v\in\Omega$ such that $d(v,\LX)<s(\epsilon)$ and choose $v^\delta\in\Omega^\delta$ such that $v^\delta\to v$. Combining Lemma~\ref{lem::harm} and Lemma~\ref{lem::convobservable}, by letting $\delta\to 0$ and then letting $v\to\LX$, we have
\[\lim_{\delta\to 0}\frac{\lambda_{c^\delta,d^\delta}^{\delta}(\LX_{\delta})}{\lambda^\delta(\LX_{\delta})}=\frac{\partial_n \Im h(z,\varphi(c),\varphi(d);\by)|_{z=\varphi(\LX)}}{\frac{1}{K(\by)}\partial_n\Im f(z; \by)|_{z=\varphi(\LX)}}.\]
Combining~\eqref{def::PH}, we have
\begin{align}\label{eqn::conditionallaw}
\lim_{\delta\to 0}\E[\LY_\delta\in(c^\delta d^\delta)\cond \LX_\delta]=&\frac{K(\by) \partial_n \Im h(z,\varphi(c),\varphi(d);\by)|_{z=\varphi(\LX)}}{\partial_n\Im f(z; \by)|_{z=\varphi(\LX)}}\notag\\
=&\int_{\varphi(c)}^{\varphi(d)}\ud x\frac{\partial_n P_{\HH}(z,x;\by)|_{z=\varphi(\LX)}}{\PartF(\varphi(\LX); \by)}.
\end{align} 
Combining with Propsition~\ref{prop::startingpoint}, we complete the proof.
\end{proof}
 
\subsection{Proof of Lemma~\ref{lem::probuniqueness}}
\label{subsec::holomorphic}
Lemma~\ref{lem::probuniqueness} looks innocent, but its proof involves complicated analysis. We will complete its proof in this section. Moreover, we will derive a symmetry property of $P_{\HH}$ which will be used in Section~\ref{subsec::cvgloewner}. 

Fix $N\ge 3$ and $y_1<\cdots<y_{2N-1}<y_{2N}$ and $x\in(y_1,y_2)\cup(y_{2N-1},y_{2N})$. Consider the differentials 
\begin{equation}\label{eqn::differential_var}
\tilde\omega_{r}:=\frac{u^{r}\ud u}{\prod_{j=1}^{2N}(u-y_{j})^{1/2}(u-x)^2},\quad \text{for }0\le r\le N-2.
\end{equation}
We choose a branch for each $\tilde{\omega}_r$ simultaneously positive on~\eqref{eqn::branchchoice}. 
Let $\{\tilde{R}(i,j): 1\le i, j\le N-2\}$ be the matrix given by 
\begin{equation}\label{eqn::R_var}
\tilde{R}(i,j)=\int_{y_{2i+1}}^{y_{2i+2}} \tilde{\omega}_{j-1},\quad\text{for }1\le i,j\le N-2. 
\end{equation}
Using Vandermonde determinant, we have
\[\det\tilde R=\int_{y_{3}}^{y_{4}}\cdots\int_{y_{2N-3}}^{y_{2N-2}}\prod_{1\le i<j\le N-2}(u_j-u_i)\prod_{i=1}^{N-2}\frac{\ud u_i}{\prod_{j=1}^{2N}(u_i-y_j)^{1/2}(u_i-x)^2}\neq 0. \]
In particular, the matrix $\tilde{R}$ in~\eqref{eqn::R_var} is invertible. Define 
\begin{equation}\label{eqn::nu_var}
(\tilde{\nu}_{N-2}, \ldots, \tilde{\nu}_1)^t: =-\tilde{R}^{-1}\left(\int_{y_3}^{y_4}\tilde{\omega}_{N-2},\,\cdots,\,\int_{y_{2N-3}}^{y_{2N-2}}\tilde{\omega}_{N-2}\right)^{t}.
\end{equation}
Define the polynomial
\begin{equation}\label{eqn::polynomial_var}
\tilde{Q}(w):=w^{N-2}+\sum_{n=0}^{N-3}\tilde{\nu}_{N-2-n} w^{n}.
\end{equation}
Let $\tilde{\mu}^{(1)}, \ldots, \tilde{\mu}^{(N-2)}$ be the $N-2$ roots of $\tilde{Q}$. 
For $u\in\HH$ and $x\in(y_1,y_2)\cup(y_{2N-1},y_{2N})$, define
\begin{equation}\label{eqn::auxfunctiong}
g(u,x;\by):=\frac{\prod_{\ell=1}^{N-2} \left(u-\tilde\mu^{(\ell)}\right)}{\prod_{j=1}^{2N}(u-y_j)^{1/2}}\quad\text{ and }\quad L(x;\by):=\frac{\partial_ug(u,x;\by)|_{u=x}}{g(x,x;\by)}.
\end{equation}
We choose a branch for $g(\cdot,x;\by)$ so that it is real and positive on~\eqref{eqn::branchchoice}. 
\begin{lemma}\label{lem::aux}
 Fix $N\ge 3$ and $y_1<\cdots<y_{2N}$. 
Define, for $z\in\HH$ and $x\in (y_{2N-1}, y_{2N})$, 
\begin{equation}\label{eqn::defV}
V(z,x;\by):=\int_{y_1}^z \frac{g(u,x;\by)\ud u}{\pi g(x,x;\by)(u-x)^2},
\end{equation}
Then,  the function $V(\cdot,x;\by)$ satisfies the following boundary conditions:
\begin{enumerate}[label=(\arabic*)]
\item \label{item::aux_bc1}
$\Im V(\cdot,x;\by)$ is positive and constant on $(y_{2i-1},y_{2i})$ for $1\le i\le N-1$; 
\item \label{item::aux_bc2}
$\Re V(\cdot,x;\by)=0$ on $(y_{2N},+\infty)\cup (-\infty,y_1)$ and $\Re V(\cdot,x;\by)$ equals the same constant on $\cup_{i=1}^{N-1}(y_{2i},y_{2i+1})$;
\item \label{item::aux_bc3}
$\Im\left(V(z,x;\by)+\frac{1}{\pi}\frac{1}{z-x}\right)$ is bounded when $z\to x$. 
\item\label{item::aux_bc4}
$\Im\left(V(z,x;\by)+\frac{1}{\pi}\frac{1}{z-x}\right)$ is constant on $(y_{2N-1},x)$ and is constant on $(x, y_{2N})$. Denote these two constants by $K_{N,1}(x;\by)$ and $K_{N,2}(x;\by)$ respectively. Then, we have 
\begin{equation}\label{eqn::consrela}
K_{N,1}(x;\by)=K_{N,2}(x;\by)+L(x;\by).
\end{equation}
\end{enumerate}
\end{lemma}

\begin{proof}
The boundary conditions~\ref{item::aux_bc1} and~\ref{item::aux_bc2} are clear from the definition and we only need to show~\ref{item::aux_bc3} and~\ref{item::aux_bc4}. It is clear that $\Im\left(V(z,x;\by)+\frac{1}{\pi}\frac{1}{z-x}\right)$ is constant on $(y_{2N-1}, x)$ and is constant on $(x, y_{2N})$ since $g(\cdot,x;\by)$ is pure imaginary on $(y_{2N-1}, y_{2N})$. By Taylor expansion, we have
\begin{align}\label{eqn::Tay}
\Im\left(V(z,x;\by)+\frac{1}{\pi}\frac{1}{z-x}\right)
=K_{N,1}(x;\by)+\frac{L(x;\by)}{\pi}\Im\log\left(\frac{z-x}{y_{2N-1}-x}\right)+O(|z-x|). 
\end{align}
This implies the boundary condition~\ref{item::aux_bc3}. By the same computation as~\eqref{eqn::Tay}, we have
\begin{equation*}
\Im\left(V(z,x;\by)+\frac{1}{\pi}\frac{1}{z-x}\right)=K_{N,2}(x;\by)+\frac{L(x;\by)}{\pi}\Im\log\left(\frac{z-x}{y_{2N}-x}\right)+O(|z-x|). 
\end{equation*}
Combining with~\eqref{eqn::Tay}, we obtain~\eqref{eqn::consrela}. This implies the boundary condition~\ref{item::aux_bc4} and completes the proof. 
\end{proof}
\begin{lemma}
For $y_1<\cdots<y_{2N-1}<x<y_{2N}$ and $z\in\HH$, recall that $K(\by)$ denotes the height of the rectangle $f(\HH)$ and $L(x;\by)$ is defined in~\eqref{eqn::auxfunctiong} and $V(z,r;\by)$ is defined in~\eqref{eqn::defV}. Define 
\begin{equation}\label{eqn::aux1}
U(z,x;\by):=\int_{x}^{y_{2N}}\ud r\left(V(z,r;\by)-\frac{K_{N,1}(r;\by)}{K(\by)}f(z;\by)\right)\exp\left(\int_{x}^{r}L(s;\by)\ud s\right).
\end{equation}
Then, the function $U(\cdot,x;\by)$ satisfies the following boundary conditions:
\begin{enumerate}[label=(\arabic*)]
\item\label{item::1}
$\Re U=0$ on $(y_{2N},+\infty)\cup (-\infty,y_1)$; 
\item\label{item::2}
$\Re U$ equals the same constant on $\cup_{i=1}^{N-1}(y_{2i},y_{2i+1})$; 
\item\label{item::3}
$\Im U$ is constant on $(y_{2i-1},y_{2i})$ for $1\le i\le N-1$ and $\Im U=0$ on $(y_{2N-1},x)$;
\item\label{item::4}
$\Im U=1$ on $(x,y_{2N})$;
\item\label{item::5}
$\Im U$ is bounded on $\overline\HH$ and continuous on $\overline\HH\setminus\{x\}$.
\end{enumerate}
\end{lemma}
\begin{proof}
From the boundary conditions of $V$ and $f$, we see that the boundary conditions~\ref{item::1},~\ref{item::2} and~\ref{item::3} hold. It remains to show~\ref{item::4} and~\ref{item::5}.

For~\ref{item::4}, for every $w\in (x,y_{2N})$, we have
\begin{align*}
&\lim_{z\to w}\Im U(z,x;\by)\\=&\lim_{\epsilon\to 0}\lim_{z\to w}\int_{w-\epsilon}^{w+\epsilon}\ud r\Im\left(V(z,r;\by)+\frac{1}{\pi}\frac{1}{z-r}-\frac{K_{N,1}(r;\by)}{K(\by)}f(z;\by)\right)\exp\left(\int_{x}^{r}L(s;\by)\ud s\right)\\
&+\lim_{\epsilon\to 0}\lim_{z\to w}\int_{(x,w-\epsilon)\cup(w+\epsilon,y_{2N})}\ud r\Im\left(V(z,r;\by)+\frac{1}{\pi}\frac{1}{z-r}-\frac{K_{N,1}(r;\by)}{K(\by)}f(z;\by)\right)\exp\left(\int_{x}^{r}L(s;\by)\ud s\right)\\
&-\frac{1}{\pi}\lim_{z\to w}\int_{x}^{y_{2N}}\ud r\Im\frac{1}{z-r}\exp\left(\int_{x}^{r}L(s;\by)\ud s\right).
\end{align*}
Combining with the boundary conditions~\ref{item::aux_bc3} and~\ref{item::aux_bc4} in Lemma~\ref{lem::aux}, we have
\[\lim_{z\to w}\Im U(z,x;\by)=-\int_x^w\ud rL(r;\by)\exp\left(\int_{x}^{r}L(s;\by)\ud s\right)+\exp\left(\int_{x}^{w}L(s;\by)\ud s\right)=1.\]
This completes the proof of~\ref{item::4}. 

For~\ref{item::5}, from the boundary conditions of $V$ and $f$, we have that $\Im U$ is bounded on $\HH\setminus\{y_{2N}\}$ and continuous on $\HH\setminus\{x,y_{2N}\}$. It remains to show that $\Im U$ is continuous near $y_{2N}$. By the explict form of $V(z,x;\by)$ in~\eqref{eqn::defV}, we have
\begin{align*}
&\Im U(z,x;\by)\notag\\=&1+\left(\Im U(z,x;\by)-\lim_{\epsilon\to 0}\Im U(y_{2N}-\epsilon,x;\by)\right)\notag\\
=&1+\int_{x}^{y_{2N}}\ud r\Im\left(\int_{y_{2N}}^z \ud u\frac{g(u,r;\by)}{\pi g(r,r;\by)(u-r)^2}-\frac{K_{N,1}(r;\by)}{K(\by)}\left(f(z;\by)-f(y_{2N};\by)\right)\right)\exp\left(\int_{x}^{r}L(s;\by)\ud s\right).
\end{align*}
By dominated convergence theorem, we have
\[\lim_{z\to y_{2N}}\Im U(z,x;\by)=1.\]
This proves the continuity of $\Im U$ and completes the proof.
\end{proof}
\begin{proof}[Proof of Lemma~\ref{lem::probuniqueness}]
Let $U(z, x; \by)$ be defined in~\eqref{eqn::aux1}.
Consider $y_1<\cdots<y_{2N-1}<t:=\varphi(c)<s:=\varphi(d)<y_{2N}$. 
Note that 
\begin{align}\label{eqn::hitconformal}
h(z,t,s;\by):=U(z,s;\by)-U(z,t;\by)
\end{align}
satisfies all the required boundary conditions. This finishes proof of the existence part.

We still need to prove the uniqueness. Suppose $k$ is another holomorphic function satisfies these boundary conditions. Define $v:=\Im(h-k)$. It suffices to show that $v=0$. Note that 
\begin{itemize}
\item
$v$ is constant on $(y_{2i-1},y_{2i})$ for $1\le i\le N$; 
\item $v=0$ on $(y_1,y_{2})\cup(y_{2N-1},y_{2N})$; 
\item $\partial_n v=0$ on $\cup_{i=1}^{N}(y_{2i},y_{2i+1})$ for $1\le i\le N$; 
\item $\int_{y_{2i-1}}^{y_{2i}}\partial_n v(z)\ud z=0$ for $2\le i\le N-1$.
\end{itemize}
By reflection with respect to $\R$, we can extend $v$ to a continuous function on $\C$ and we still denote this extension by $v$. By the third boundary condition, $v$ is harmonic on $\C\setminus\cup_{i=1}^{N}(y_{2i-1},y_{2i})$. It suffices to show that $v$ only gets its maximum and minimum on $(y_{1},y_{2})\cup(y_{2N-1},y_{2N})$. By maximum principle, $v$ gets its maximum only on $\cup_{i=1}^{N}(y_{2i-1},y_{2i})$. For $2\le i\le N-1$, if there exists $z\in (y_{2i-1},y_{2i})$ such that $\partial_n v(z)>0$, then $v$ can not get its maximum on $(y_{2i-1},y_{2i})$. If $\partial_n v(z)\le 0$ for every $z\in(y_{2i-1},y_{2i})$, by the fourth boundary condition, we have $\partial_n v=0$ on $(y_{2i-1},y_{2i})$. This implies that $v$ is harmonic on $(y_{2i-1},y_{2i})$. Thus, $v$ can not get its maximum on $(y_{2i-1},y_{2i})$. Therefore $v$ only gets its maximum on $(y_{1},y_{2})\cup(y_{2N-1},y_{2N})$. By the same argument, $v$ only gets its minimum on $(y_{1},y_{2})\cup(y_{2N-1},y_{2N})$. This implies $v=0$ and completes the proof of uniqueness.
\end{proof}

In the following, we will derive a symmetry property of $P_{\HH}$. Recall from~\eqref{def::PH} that we have defined $P_{\HH}(z, x; \by)$ with $x\in (y_{2N-1}, y_{2N})$. To state the symmetry property, we first generalize this definition to $x\in (y_1, y_2)$.  The following lemma is the analogue of Lemma~\ref{lem::probuniqueness}. 

\begin{lemma}\label{lem::htilde}
For any $(t,s)\subset(y_{1},y_{2})$, there exists a unique holomorphic function $\tilde h$ on $\HH$, which satisfies the following boundary conditions:
\begin{enumerate}[label=(\arabic*)]
\item
$\Re \tilde{h}=0$ on $(y_{2N},+\infty)\cup (-\infty,y_1)$; 
\item  $\Re \tilde{h}$ equals the same constant on $\cup_{i=1}^{N-1}(y_{2i},y_{2i+1})$; 
\item $\Im \tilde{h}$ is constant on $(y_{2i-1},y_{2i})$ for $2\le i\le N$; 
\item $\Im \tilde{h}=0$ on $(y_{2N-1},y_{2N})\cup (y_{1},t)\cup(s,y_{2})$ and $\Im \tilde{h}=1$ on $(t,s)$;
\item
$\Im \tilde{h}$ is bounded on $\overline\HH$ and continuous on $\overline\HH\setminus\{t,s\}$.
\end{enumerate}
\end{lemma}
\begin{proof}
Recall that $g,L$ are defined in~\eqref{eqn::auxfunctiong}. For $w\in (y_1, y_2)$, define 
\[\tilde{V}(z,w;\by):=\int_{y_{2N}}^z \frac{g(u,w;\by)\ud u}{\pi g(w,w;\by)(u-w)^2}.\] 
Using the same argument as in the proof of Lemma~\ref{lem::aux}, we have that $\Im\left(\tilde{V}(z,w;\by)+\frac{1}{z-w}\right)$ is constant on $(y_1,w)$. We denote this constant by $\tilde{K}_{N,1}(w;\by)$. Recall that $K(\by)$ denotes the height of the rectangle $f(\HH)$. Define 
\begin{align}\label{eqn::realhitconformal}
\tilde{U}(z,w;\by)\notag
:=\int_{w}^{y_{2}}\ud r\left(V(z,r;\by)-\frac{\tilde{K}_{N,1}(r;\by)}{K(\by)}f(z;\by)\right)\exp\left(\int_{w}^{r}L(s;\by)\ud s\right).
\end{align}
By the same argument in the proof of Lemma~\ref{lem::probuniqueness}, the following holomorphic function 
\[\tilde{h}(z,t,s;\by):=\tilde{U}(z,s;\by)-\tilde{U}(z,t;\by)\]
satisfies all the required boundary conditions.
The uniqueness can be proved by the same argument as in the proof of Lemma~\ref{lem::probuniqueness}.
\end{proof}

Let $\tilde{h}$ be the holomorphic function in Lemma~\ref{lem::htilde}. For $w\in (y_{1},y_{2})$, the following limit exists: 
\begin{equation}\label{def::PH_sym}
P_{\HH}(z,w;\by):=K(\by)\lim_{t,s\to w}\frac{\Im \tilde{h}(z,t,s;\by)}{s-t}. 
\end{equation} 

Now, we are ready to state the symmetry property of $P_{\HH}$. 
\begin{lemma}\label{lem::sym}
Let $P_{\HH}$ be the function defined in~\eqref{def::PH} and~\eqref{def::PH_sym}. 
For any $y_1<w<y_2<\ldots<y_{2N-1}<x<y_{2N}$, we have
\[\partial_n P_{\HH}(z,x;\by)|_{z=w}=\partial_n P_{\HH}(z,w;\by)|_{z=x}.\]
\end{lemma}
\begin{proof}
This is from the reversibility of loop-erased random walk and Proposition~\ref{prop::jointdis}.
\end{proof}

\subsection{Proof of Proposition~\ref{prop::lawofgamma}}
\label{subsec::cvgloewner}

Let $P_{\HH}$ be the function defined in~\eqref{def::PH} and~\eqref{def::PH_sym}. 
For $y_1<w<y_2<\cdots<y_{2N-1}<x<y_{2N}$, define 
\begin{equation}\label{eqn::observable}
F(w, x;\by):=\frac{\partial_n P_{\HH}(z,w;\by)|_{z=x}}{\PartF(w; \by)}.
\end{equation}
By Lemma~\ref{lem::sym}, we also have
\begin{equation}\label{eqn::observablesymmetry}
F(w, x;\by)=\frac{\partial_n P_{\HH}(z,x;\by)|_{z=w}}{\PartF(w; \by)}.
\end{equation}

Recall that $\varphi$ is a fixed conformal map from $\Omega$ onto $\HH$ and we denote $y_j=\varphi(x_j)$ for $1\le j\le 2N$. We also fix conformal maps $\varphi_{\delta}: \Omega^{\delta}\to\HH$ such that $\varphi_{\delta}^{-1}$ converges to $\varphi^{-1}$ locally uniformly and that $\varphi_{\delta}(x_j^{\delta})\to\varphi(x_j)$ for $1\le j\le 2N$. From Lemma~\ref{lem::tight}, the family $\{\gamma_{\delta}\}_{\delta>0}$ is tight. Let $\gamma$ be any subsequential limit. To simplify the notation, we still denote the convergent subsequence by $\{\gamma_{\delta}\}_{\delta>0}$. We couple $\{\gamma_{\delta}\}_{\delta>0}$ and $\gamma$ together so that $\gamma_{\delta}\to\gamma$ almost surely. 
The proof of Proposition~\ref{prop::lawofgamma} follows the strategy in~\cite{HanLiuWuUST}: We first show that the function $F$ in~\eqref{eqn::observable} gives a martingale observable for $\gamma$, see Lemma~\ref{lem::mart_observable}; then we show that $F$ satisfies a certain PDE, see Lemma~\ref{lem::observable_PDE}. 
With these two lemmas at hand, we solve the driving function of $\gamma$ from the martingale observable and complete the proof of Proposition~\ref{prop::lawofgamma}. In the proof of Proposition~\ref{prop::lawofgamma}, we still need a technical lemma: Lemma~\ref{lem::aux2}.

We parameterize $\gamma$ such that $\varphi(\gamma)$ is parameterized by its half-plane capacity and we parameterize $\gamma_\delta$ similarly. 
For the continuous curve $\varphi(\gamma)$, we denote by $(W_t, t\ge 0)$ its driving function and by $(g_t, t\ge 0)$ the corresponding conformal maps. 

\begin{lemma}\label{lem::mart_observable}
For any $x\in(y_{2N-1},y_{2N})$, the process 
\[\left(g_t'(x)F(W_t, g_t(x);g_t(y_1),\ldots,g_t(y_{2N})),\, t\ge 0\right)\] 
is a martingale up to the first time that $\gamma$ hits $\cup_{i=2}^{N}(x_{2i-1}x_{2i})$. 
\end{lemma}
\begin{proof}
Recall that $\tilde{\gamma}_{\delta}$ and $\LY_{\delta}$ are defined in~\eqref{eqn::def_twoends}. Recall that $\gamma_\delta$ equals $\tilde{\gamma}_\delta$ before the first time they hit $\cup_{i=2}^{N}(x_{2i-1}^\delta x_{2i}^\delta)$. For $(cd)\subset(x_{2N-1}x_{2N})$, choose the discrete approximation $(c^\delta d^\delta)\subset (x_{2N-1}^\delta x_{2N}^\delta)$ such that $(c^\delta d^\delta)$ converges to $(cd)$ in metric~\eqref{eqn::curves_metric}. Define
\[M^\delta_t:=\E\left[\LY_\delta\in(c^\delta d^\delta)\cond \gamma_\delta[0,t]\right].\]
By~\eqref{eqn::conditionallaw}, we have
\[M_0^\delta\to M_0:=\int_{\varphi(c)}^{\varphi(d)}\ud x\frac{\partial_n P_{\HH}(z, x;\by)|_{z=\varphi(\LX)}}{\PartF(\varphi(\LX); \by)}.\]
Since $\gamma_\delta$ converges to $\gamma$ in metric~\eqref{eqn::curves_metric}, by the Markov property of $\gamma_\delta$, we have
\[
M^\delta_t\to M_t:=\int_{g_t(\varphi(c))}^{g_t(\varphi(d))}\ud x\frac{\partial_n P_{\HH}(z,x;g_t(y_1),\ldots,g_t(y_{2N}))|_{z=W_t}}{\PartF(W_t;g_t(y_1),\cdots,g_t(y_{2N}))}.
\]
Since $M_t^\delta$ is a discrete martingale up to the first time that $\gamma_\delta$ hits $\cup_{i=2}^{N}(x^\delta_{2i-1}x^\delta_{2i})$, the process $(M_t, t\ge 0)$ is a continuous martingale up to the first time that $\gamma$ hits $\cup_{i=2}^{N}(x_{2i-1}x_{2i})$. See more details in~\cite[Lemma~5.10]{HanLiuWuUST}. This implies that the process
\[\frac{g'_t(x)\partial_n P_{\HH}(z,g_t(x);g_t(y_1),\ldots,g_t(y_{2N}))|_{z=W_t}}{\PartF(W_t;g_t(y_1),\cdots,g_t(y_{2N}))}\] 
is a continuous martingale for all $x\in (y_{2N-1}, y_{2N})$ up to the hitting time. Combining with~\eqref{eqn::observablesymmetry}, the process
$F(W_t, g_t(x);g_t(y_1),\ldots,g_t(y_{2N}))$ is a continuous martingale for all $x\in (y_{2N-1}, y_{2N})$ up to the hitting time as desired.
\end{proof}
\begin{lemma}\label{lem::observable_PDE}
For $y_1<w<y_2<\ldots<y_{2N-1}<x<y_{2N}$, we have
\begin{equation}\label{eqn::Poisson_partialn_PDE}
\left(\sum_{i=1}^{2N}\frac{2}{y_i-w}\partial_{y_i}+2\frac{\partial_w\PartF(w; \by)}{\PartF(w; \by)}\partial_w+\partial_w^2+\frac{2}{x-w}\partial_x+\frac{-2}{(x-w)^2}\right)F=0.
\end{equation}
\end{lemma}
\begin{proof}
The proof is similar to the proof of~\cite[Corollary 5.9]{HanLiuWuUST}. We summarize it below briefly. Define
\[\LD:=\sum_{i=1}^{2N}\frac{2}{y_i-w}\partial_{y_i}+2\frac{\partial_w\PartF(w; \by)}{\PartF(w; \by)}\partial_w+\partial_w^2+\Re\left(\frac{2}{z-w}\right)\partial_x+\Im\left(\frac{2}{z-w}\right)\partial_y.\]
Define \[P(z):=\frac{P_\HH(z,w;\by)}{\PartF(w; \by)},\quad \text{for }z\in\HH.\]
It suffices to show that $\LV:=\LD P=0$. By direct computation, $\LV$ is continuous on $\overline\HH\setminus\{w\}$ and there exist two holomorphic functions $G_1$ and $G_2$ such that
\[\LV(z)=\Im\left(G_1(z)+\frac{G_2(z)}{z-w}\right).\]
Moreover, $\LV$ satisfies the following boundary conditions:
\begin{itemize}
\item
$\partial_n \LV=0$ on $\R\setminus\cup_{i=1}^{N}(y_{2i-1},y_{2i})$ and $\int_{y_{2i-1}}^{y_{2i}}\partial_n \LV \ud z=0$ for $2\le i\le N-1$; 
\item $\LV=0$ on $(y_{2N-1},y_{2N})$;
\item
$\LV(z)+G_2(w)\Im\frac{1}{z-w}=0$ on $(y_1,y_2)$.
\end{itemize}
From the definition, $P$ also satisfies the first boundary condition and $P(z)+\frac{ K(\by)}{\pi \PartF(w; \by)}\Im\frac{1}{z-w}$ is bounded when $z\to w$. By the similar argument in the proof of Lemma~\ref{lem::probuniqueness}, we have
\[\LV(z)+\frac{\pi G_2(w)\PartF(w; \by)}{K(\by)}P(z)=0.\] 
But by direct computation, we have
\[\int_{y_{2N-1}}^{y_{2N}}\ud x\partial_n P(z)|_{z=x}=1,\quad\int_{y_{2N-1}}^{y_{2N}}\ud x\partial_n \LV(z)|_{z=x}=0.\]
This implies that $G_2(w)=0$. Thus, $\LV$ is bounded on $\overline\HH$. By the same argument in the proof of Lemma~\ref{lem::probuniqueness}, we have $\LV=0$ and this completes the proof.
\end{proof}
\begin{lemma}\label{lem::aux2}
Recall that $L$ is defined in~\eqref{eqn::auxfunctiong}. Suppose $J(\cdot)$ is smooth on  $(y_{2N-1},y_{2N})$ and is not identically zero. Then, for fixed $y_1<w<y_2<\ldots<y_{2N-1}<y_{2N}$, the function
\[v(x):=J(x)+L(x;\by)\int_x^{y_{2N}}\ud r J(r)\exp\left(\int_x^rL(s;\by)ds\right)\]
is not identically zero on $(y_{2N-1},y_{2N})$.
\end{lemma}
\begin{proof}
Note that $L(\cdot;\by)$ has finitely many roots on $(y_{2N-1},y_{2N})$. We denote all the roots by $z_1<\ldots<z_n$. Define $z_0:=y_{2N-1}$ and define $z_{n+1}:=y_{2N}$. Suppose the function $v$ is identically zero. Then, for $0\le i\le n$ and for all $x\in(z_i,z_{i+1})$, we have
\begin{equation}\label{eqn::nozero}
-\frac{J(x)}{L(x;\by)}=\int_x^{y_{2N}}\ud r J(r)\exp\left(\int_x^rL(s;\by)ds\right).
\end{equation}
This implies 
\[\partial_x\left(-\frac{J(x)}{L(x;\by)}\right)=-J(x)-L(x;\by)\int_x^{y_{2N}}\ud r J(r)\exp\left(\int_x^rL(s;\by)ds\right)=0.\]
Thus, there exists a constant $C_i$ such that 
\[J(x)=C_iL(x;\by), \quad \text{for }x\in (z_i, z_{i+1}).\]
By the continuity, we have $C_i=C$ for all $0\le i\le n$.
Plugging into~\eqref{eqn::nozero}, we have
\[-C=C\int_x^{y_{2N}}\ud r L(r;\by)\exp\left(\int_x^rL(s;\by)ds\right)=C\left(1-\exp\left(\int_x^{y_{2N}}L(s;\by)ds\right)\right),\]
which is a contradiction.
\end{proof}
\begin{proof}[Proof of Proposition~\ref{prop::lawofgamma}]
Recall that $g$ is defined in~\eqref{eqn::auxfunctiong} and $K(\by)$ denotes the height of the rectangle $f(\HH)$. For any $y_1<w<y_2<\ldots<y_{2N-1}<x<y_{2N}$, define
\[J(w,x;\by):=\partial_w\left(\frac{K(\by)g(w,x;\by)}{\pi g(x,x;\by)\PartF(w; \by)(w-x)^2}\right).\]
Note that $J(w,\cdot;\by)$ is not identically zero.
By the explicit form of $F$, we have
\begin{equation*}
\partial_wF(w,x;y_1,\ldots,y_{2N})=-J(w,x;\by)-L(x;\by)\int_x^{y_{2N}}\ud r J(w,r;\by)\exp\left(\int_x^rL(s;\by)ds\right).
\end{equation*}
By Lemma~\ref{lem::aux2}, there exists $x\in(y_{2N-1}, y_{2N})$ such that $\partial_wF(w,x;\by)\neq 0$. By implicit function theorem, $w$ is locally a smooth function of $(F,x,\by)$. Thus, we can choose deterministic $x$, such that $W_t$ is locally a smooth function of $(M_t,g_t(x),g_t(y_1),\ldots,g_t(y_{2N}))$. By It\^{o}'s formula, $(W_t,t\ge 0)$ is a semimartingale. See more detials in~\cite[Proof of Theorem 1.6]{HanLiuWuUST} . Denote by $R_t$ the drift term of $W_t$. Combining It\^{o}'s formula and~\eqref{eqn::Poisson_partialn_PDE}, we have
\begin{align}\label{eqn::final}
\partial_w F\left(\ud R_t-2\frac{\partial_w\PartF(W_t;g_t(y_1),\ldots,g_t(y_{2N}))}{\PartF(W_t;g_t(y_1),\ldots,g_t(y_{2N}))}\ud t\right)+\frac{1}{2}\partial_w^2 F\left(\ud \langle W\rangle_t-2\ud t\right)=0. 
\end{align}
It suffices to show that for any $\tilde{y}_1<m<\tilde{y}_2<\ldots<\tilde{y}_{2N}$, there exist $\tilde{x}_1,\tilde {x}_2\in(\tilde{y}_{2N-1},\tilde{y}_{2N})$, such that 
\begin{equation}\label{eqn::neq}
\partial_m F(m,\tilde{x}_1;\tilde{\by})\partial^2_m F(m,\tilde{x}_2;\tilde{\by})\neq \partial_m F(m,\tilde{x}_2;\tilde{\by})\partial^2_m F(m,\tilde{x}_1;\tilde{\by}),\quad \text{where }\tilde{\by}=(\tilde{y}_1, \ldots, \tilde{y}_{2N}). 
\end{equation}
Assuming this is true. 
Then, by setting $\tilde{y}_i=g_t(y_i)$ for $1\le i\le N$ and $m=W_t$ and $\tilde{x}_i=g_t^{-1}(x_i)$ for $i=1,2$ and plugging into~\eqref{eqn::final}, we have
\[\ud R_t=2\frac{\partial_w\PartF(W_t;g_t(y_1),\ldots,g_t(y_{2N}))}{\PartF(W_t;g_t(y_1),\ldots,g_t(y_{2N}))}\ud t\quad\text{and}\quad \ud \langle W\rangle_t=2\ud t\]
as desired.

It remains to show~\eqref{eqn::neq}. If this is not the case, there exists a continuous function $C(m;\tilde{\by})$ such that 
\[\partial_m^2F(m,\tilde{x};\tilde{\by})=C(m;\tilde{\by})\partial_m F(m,\tilde{x};\tilde{\by}),\quad \text{for all }\tilde{x}\in(\tilde{y}_{2N-1},\tilde{y}_{2N}).\]
This implies that, for all $m\in(\tilde{y}_{2N-1},\tilde{y}_{2N})$, 
\begin{align*}
&(C(m;\tilde{\by})J(m,\tilde{x};\tilde{\by})-\partial_mJ(m,\tilde{x};\tilde{\by}))\\
&+L(\tilde{x};\tilde{\by})\int_{\tilde{x}}^{\tilde{y}_{2N}}\ud r (C(m;\tilde{\by})J(m,\tilde{x};\tilde{\by})-\partial_mJ(m,\tilde{x};\tilde{\by}))\exp\left(\int_{\tilde{x}}^rL(s;\tilde{\by})ds\right)=0.
\end{align*}
By Lemma~\ref{lem::aux2}, we have
\[C(m;\tilde{\by})J(m,\tilde{x};\tilde{\by})-\partial_mJ(m,\tilde{x};\tilde{\by})=0,\quad\text{for all }m\in(\tilde{y}_{2N-1},\tilde{y}_{2N}).\]
Define 
\[S(m,\tilde{x};\tilde{\by}):=\partial_m\left(\frac{g(m,\tilde{x};\tilde{\by})}{\PartF(m; \tilde{\by})(m-\tilde{x})^2}\right).\]
This implies that 
\begin{equation}\label{eqn::need}
C(m;\tilde{\by})S(m,\tilde{x};\tilde{\by})-\partial_mS(m,\tilde{x};\tilde{\by})=0,\quad\text{for all }m\in(\tilde{y}_{2N-1},\tilde{y}_{2N}).
\end{equation}
From the construction in~\eqref{eqn::auxfunctiong}, $g(m,\cdot;\tilde{\by})$ can be holomorphically extended to 
\[O:=\C\setminus\left((\tilde{y}_2,\tilde{y}_{2N-1})\cup_{i=1,2N}\{\tilde{y}_i+\ii t:t\in[0,+\infty)\}\right).\] Thus, the following function is holomorphic on $O$: 
\[k(\cdot, m;\tilde{\by}):=C(m;\tilde{\by})S(m,\cdot;\tilde{\by})-\partial_mS(m,\cdot;\tilde{\by}).\] 
Thus, by~\eqref{eqn::need}, the function $k(\cdot,m;\tilde{\by})$ is identically zero.
This implies that~\eqref{eqn::need} holds for $z\in O$. By letting $z\to +\infty$, we have $C(m;\tilde{\by})=0$. But, by letting $z\to m$, we have $\partial_mS(m,z;\tilde\by)\to \infty$. This is a contradiction.
This completes the proof.
\end{proof}

\begin{proof}[Proof of Theorem~\ref{thm::lerwconv}]
Corollary~\ref{coro::LX} implies the first statement and Proposition~\ref{prop::lawofgamma} implies the second statement.
\end{proof}

We end this section by a discussion on the connection between~\eqref{eqn::lerw_loewner} and $\SLE_{\kappa}(\rho)$ process. From~\eqref{eqn::conformalmap_slitrect}, we have 
\[\frac{f''(z)}{f'(z)}=\sum_{\ell=1}^{N-2}\frac{1}{z-\mu^{(\ell)}}-\frac{1}{2}\sum_{j=1}^{2N}\frac{1}{z-y_j}. \]
Plugging into~\eqref{eqn::lerw_loewner}, we see that the conditional law of $\varphi(\gamma)$ given $\LX$ is the Loewner chain with the following driving function 
\begin{equation}\label{eqn::lerwloewner_aux}
\ud W_t=\sqrt{2}\ud B_t+\left(\sum_{\ell=1}^{N-2}\frac{2}{W_t-\mu^{(\ell)}_t}-\sum_{j=1}^{2N}\frac{1}{W_t-g_t(y_j)}\right)\ud t,
\end{equation}
where $\mu^{(\ell)}=\mu^{(\ell)}(y_1, \ldots, y_{2N})$ are parameters in~\eqref{eqn::conformalmap_slitrect} and Lemma~\ref{lem::parameters} and 
\[\mu^{(\ell)}_t=\mu^{(\ell)}(g_t(y_1), \ldots, g_t(y_{2N})).\]

When $N=2$, Eq.~\eqref{eqn::lerwloewner_aux} becomes 
\[\ud W_t=\sqrt{2}\ud B_t+\sum_{j=1}^{4}\frac{-1}{W_t-g_t(y_j)}\ud t.\]
This is the same as $\SLE_2(-1,-1; -1,-1)$ as proved in~\cite{HanLiuWuUST}. When $N\ge 3$, the term $\mu^{(\ell)}_t$ plays an essential role in~\eqref{eqn::lerwloewner_aux} and the law of $(W_t, t\ge 0)$ does not belong to the family of $\SLE_2(\rho)$ process. 